\documentclass[reqno,a4paper,final]{amsart}

\usepackage[utf8]{inputenc}
\usepackage[T1]{fontenc}
\usepackage[english]{babel}
\usepackage{ifthen}
\usepackage{cite}
\usepackage{enumitem}
\usepackage{amsmath}
\usepackage{amsthm}
\usepackage{amsfonts}
\usepackage{amssymb}
\usepackage{dsfont}
\usepackage{mathtools}
\usepackage[unicode,bookmarksopen]{hyperref}
\numberwithin{equation}{section} 

\newtheorem{lemma}{Lemma}[section]
\newtheorem{corollary}[lemma]{Corollary}
\newtheorem{proposition}[lemma]{Proposition}
\newtheorem{theorem}[lemma]{Theorem}

\theoremstyle{definition}
\newtheorem{definition}[lemma]{Definition}

\newtheorem{remark}[lemma]{Remark}

\newlist{thm_enum}{enumerate}{1}
\setlist[thm_enum]{label=\normalfont(\alph*)}
\newlist{def_enum}{enumerate}{1}
\setlist[def_enum]{label=\normalfont(\roman*)}

\newcommand{\IZ}{\mathbb{Z}}

\newcommand{\IN}{\mathbb{N}}
\newcommand{\IR}{\mathbb{R}}
\newcommand{\IC}{\mathbb{C}}

\renewcommand{\epsilon}{\varepsilon}
\renewcommand{\phi}{\varphi}
\newcommand{\abs}[1]{\left\lvert#1\right\rvert}
\newcommand{\biggabs}[1]{\biggl\lvert#1\biggr\rvert}

\newcommand{\norm}[1]{\left\lVert#1\right\rVert}
\newcommand{\biggnorm}[1]{\biggl\lVert#1\biggr\rVert}

\newcommand{\R}[2][\empty]{
	\ifthenelse{\equal{#1}{\empty}}
		{\mathcal{R}\left\{#2\right\}}
		{\mathcal{R}_{#1}\left\{#2\right\}}
}

\DeclareMathOperator{\linspan}{span}

\DeclareMathOperator{\Id}{Id}
\DeclareMathOperator{\sign}{sign}

\DeclareMathOperator{\Rad}{Rad}

\allowdisplaybreaks[4]

\begin{document}

\title[Kalton-Lancien Revisited: Maximal Regularity does not extrapolate]{The Kalton-Lancien Theorem Revisited: Maximal Regularity does not extrapolate}

\begin{abstract}
	We give a new more explicit proof of a result by Kalton and Lancien stating that on each Banach space with an unconditional basis not isomorphic to a Hilbert space there exists a generator $A$ of a holomorphic semigroup which does not have maximal regularity. In particular, we show that there always exists a Schauder basis $(f_m)$ such that $A$ can be chosen of the form $A(\sum_{m=1}^{\infty} a_m f_m) = \sum_{m=1}^{\infty} 2^m a_m f_m$. Moreover, we show that maximal regularity does not extrapolate: we construct consistent holomorphic semigroups $(T_p(t))_{t \ge 0}$ on $L^p(\IR)$ for $p \in (1, \infty)$ which have maximal regularity if and only if $p = 2$. These assertions were both open problems. Our approach is completely different than the one of Kalton and Lancien. We use the characterization of maximal regularity by $\mathcal{R}$-sectoriality for our construction.
\end{abstract}

\author{Stephan Fackler}
\address{Institute of Applied Analysis, University of Ulm, Helmholtzstr. 18, 89069 Ulm}
\email{stephan.fackler@uni-ulm.de}
\thanks{The author was supported by a scholarship of the ``Landesgraduiertenförderung Baden-Württemberg''}
\thanks{I would like to thank the anonymous reviewer for pointing out an error in Section~\ref{section:l_p} in the first version of the manuscript and Jochen Glück for showing me how to streamline the proof of Lemma~\ref{lem:technical_lemma}.}

\maketitle

\section{Introduction}

The generator $-A$ of a strongly continuous semigroup $(T(t))_{t \ge 0}$ on a Banach space $X$ is said to have $(p,T)$-\emph{maximal regularity} ($T>0, 1 < p < \infty$) if for all $f \in L^p((0,T);X)$ the mild solution $u(t) = \int_0^t T(t-s) f(s) \, ds$ of the inhomogeneous abstract Cauchy problem
	\[ \begin{cases} 
			\dot{u}(t) + Au(t) & = f(t) \\
			u(0) & = 0.
		\end{cases}
	\]
satisfies $u \in W^{1,p}((0,T);X) \cap L^p((0,T); D(A))$. One can show that this property is independent of $p \in (1,\infty)$ and $T \in (0, \infty)$. Therefore one simply speaks of \emph{maximal regularity}. 

The property of maximal regularity has attracted the attention of many mathematicians as it is a tool to solve non-linear partial differential equations by means of linearization and applying a fixed point theorem. It was known for a long time that for $-A$ to have maximal regularity it is necessary that $-A$ generates a holomorphic semigroup~~\cite[Theorem~2.2]{Dor93}. Conversely, it is known that all generators of holomorphic semigroups on a Hilbert space have maximal regularity by a result of L. de Simon~\cite[Lemma 3,1]{Sim64}. Moreover, one can show maximal regularity for large classes of differential operators on $L^p$-spaces, e.g. generators of holomorphic semigroups satisfying Gaussian bounds. This leads H. Brézis to the question whether all generators of holomorphic semigroups on $L^p$ ($1 < p < \infty$) have maximal regularity.

\begin{definition} A complex Banach space $X$ has the \emph{maximal regularity property (MRP)} if every generator of a holomorphic semigroup on $X$ has maximal regularity.
\end{definition}

Besides Hilbert spaces $L^{\infty}([0,1])$ has (MRP) for the simple reason that every generator of a strongly continuous semigroup is already bounded by a famous result of H.P.~Lotz~\cite[Theorem~3]{Lot85}. In contrast, N.J. Kalton and G. Lancien gave the following negative answer to the maximal regularity problem in~\cite[Theorem~3.3]{KalLan00} (extensions for UMD-spaces with finite-dimensional Schauder decompositions can be found in~\cite[Theorem~3.4]{KalLan02}).

\begin{theorem}[N.J. Kalton, G. Lancien] Let $X$ be a complex Banach space with an unconditional basis that has (MRP). Then $X$ is isomorphic to $\ell^2$.
\end{theorem} 

For $p \in (1,\infty)$ the Haar basis is an unconditional basis for $L^p([0,1])$. Hence for $p \neq 2$ by Kalton-Lancien's result, there exists a generator of a holomorphic semigroup on $L^p$ not having maximal regularity. However, Kalton-Lancien's approach  only yields the pure existence of such a counterexample. Moreover, on the basis of their arguments it seems impossible to write down a counterexample explicitly. In particular in this situation, the following questions remain open:

\begin{itemize}
	\item Does there always exist a counterexample for which the generator is given as a Schauder multiplier?
	\item Does maximal regularity extrapolate? By this we mean the following question: let $(T_2(t))_{t \ge 0}$ and $(T_p(t))_{t \ge 0}$ be consistent semigroups on $L^2$ and $L^p$ ($p \in (1,\infty) \setminus \{2\}$) such that $(T_2(t))_{t \ge 0}$ is holomorphic. Does $(T_p(t))_{t \ge 0}$ have maximal regularity automatically (cf.~\cite[7.2.2]{Are04})?
\end{itemize}

We will give answers to these questions. Before we make some comments. Kalton and Lancien's proof suggests naively that a counterexample could be given as a Schauder multiplier with respect to the given unconditional basis. However, we will recall in Theorem~\ref{thm:unconditionality_implies_maxreg} that the unconditionality of the basis forces semigroups generated by Schauder multipliers on large classes of Banach spaces, in particular $L^p$-spaces with $p \in (1, \infty)$, to have maximal regularity. So a counterexample can only be found if the basis is conditional.

In this paper we will give a new proof of Kalton-Lancien's result which in particular brings light to this phenomenon. We will use the unconditional basis to construct systematically a conditional one which then can be used to write down an explicit counterexample.

Secondly, Kalton-Lancien's method only yields a counterexample for each single Banach space with an unconditional basis. In particular, it is not clear whether such a counterexample on an $L^p$-space for $p \in (1, \infty) \setminus \{2\}$ extends to a semigroup on $L^2$. Therefore one can ask whether, given a holomorphic semigroup on $L^2$ (which has maximal regularity by the result of Simon) which extrapolates to semigroups on $L^p$ for $p \in (1, \infty)$, maximal regularity always extrapolates to those $p$ as well. The Stein interpolation theorem says that this question has a positive answer if maximal regularity is replaced by the weaker property of holomorphy. In this paper we will give a negative answer to this question. Indeed, our method is explicit enough to construct consistent holomorphic semigroups on $L^p$ for $1 < p < \infty$ such that maximal regularity is violated for $p \neq 2$. Under additional assumptions, this cannot happen. For example, if the semigroup satisfies Gaussian estimates and is holomorphic in $L^2$, then maximal regularity extrapolates to all $p \in (1, \infty)$~\cite[Theorem~3.1]{HiePru97}.

Our alternative approach uses the characterization of maximal regularity by the Rademacher boundedness of the semigroup in a sector. We now introduce the necessary terminology. Let $X$ be a Banach space. A family of operators $\mathcal{T} \subset \mathcal{L}(X)$ is called \emph{$\mathcal{R}_p$-bounded} if there exists a constant $C_p > 0$ such that for all $n \in \IN$, all $T_1, \ldots, T_n \in \mathcal{T}$ and all $x_1, \ldots, x_n \in X$
	\begin{equation} \biggl( \int_0^1 \biggnorm{\sum_{k=1}^n r_k(\omega) T_kx_k}^p \, d\omega \biggr)^{1/p} \le C_p \biggl( \int_0^1 \biggnorm{\sum_{k=1}^n r_k(\omega) x_k}^p d\omega \biggr)^{1/p}, \label{eq:rademacher_bounded} \end{equation}
where $r_k(\omega) \coloneqq \sign \sin(2^k \pi \omega)$ denotes the \emph{$k$-th Rademacher function}. The closure of the finite Rademacher sums (as in \eqref{eq:rademacher_bounded}) in $L^p([0,1]; X)$ is denoted by $\Rad_p(X)$. One can show that the norms on $\Rad_p(X)$ are equivalent (the Kahane-Khintchine inequality) for all $p \in [1,\infty)$ and consequently one  writes simply $\Rad(X)$.
Hence, $\mathcal{R}_p$-boundedness holds for some $p \in [1,\infty)$ if and only if it holds for all $p$. We therefore speak simply of \emph{$\mathcal{R}$-boundedness}. The smallest constant such that~\eqref{eq:rademacher_bounded} holds is denoted by $\mathcal{R}_p(\mathcal{T})$ and one often omits the subindex if one is merely interested in the finiteness of the constant. For a detailed exposition of maximal regularity and $\mathcal{R}$-boundedness see~\cite{KunWei04} and \cite{DHP03}. We can now give the characterization of maximal regularity. For this let $\Sigma_{\delta} \coloneqq \{ z \in \IC \setminus \{0\}: \abs{\arg z} < \delta \}$ denote the sector of angle $\delta$ in the complex plane.

\begin{definition}[$\mathcal{R}$-analyticity] Let $X$ be a complex Banach space and $-A$ the generator of a holomorphic $C_0$-semigroup $(T(z))_{z \in \Sigma_{\tilde{\delta}}}$ on $X$. $(T(z))_{z \in \Sigma_{\tilde{\delta}}}$ is called \emph{$\mathcal{R}$-analytic} if there exists a $\tilde{\delta} > \delta > 0$ such that
	\[ \R{T(z): z \in \Sigma_\delta, \abs{z} \le 1} < \infty. \]
\end{definition}

\begin{theorem} Let $X$ be a complex Banach space and $-A$ the generator of a holomorphic $C_0$-semigroup $(T(z))_{z \in \Sigma}$ on $X$. Then the following hold:
	\begin{thm_enum}
		\item\label{thm:r-maxreg_implies_r-analyticity} If $-A$ has maximal regularity, then $(T(z))_{z \in \Sigma}$ is $\mathcal{R}$-analytic.
		\item\label{thm:r-analyticity_implies_maxreg} Conversely, on a UMD-space $X$, $\mathcal{R}$-analyticity of $(T(z))_{z \in \Sigma}$ implies maximal regularity.	
	\end{thm_enum}
\end{theorem}

Assertion~\ref{thm:r-maxreg_implies_r-analyticity} is due to P.~Clément and J.~Prüss~\cite[Proposition 1]{ClePru01} and assertion~\ref{thm:r-analyticity_implies_maxreg} is due to L.~Weis~\cite[Theorem~4.2]{Wei01}. A proof of both can be found in~\cite{KunWei04}. Notice that the UMD-property is only needed in the proof of part~\ref{thm:r-analyticity_implies_maxreg}. For the definition of and details on UMD-spaces see~\cite{Fra86} and \cite{Bur01}. We only remark that all $L^p$-spaces for $p \in (1, \infty)$ are UMD-spaces.

In particular, for constructing generators of holomorphic semigroups not having maximal regularity it is sufficient (independently of the UMD-property) to give examples of holomorphic semigroups which are not $\mathcal{R}$-analytic.

We outline very shortly our approach and the contents of this paper. We start by giving a short introduction into the theory of Schauder bases and explain how we can use them to construct generators of holomorphic semigroups. We then explain how we can associate to an $\mathcal{R}$-analytic semigroup a holomorphic semigroup on $\Rad(X)$, an idea which goes back to W. Arendt and S. Bu~\cite{AreBu03}. In this paper, we show that on Banach spaces with an unconditional basis we can construct semigroups out of Schauder multipliers which are holomorphic but whose associated semigroups on $\Rad(X)$ are not holomorphic, thereby giving our new proof of the Kalton-Lancien result. Before doing this in the general abstract case of a Banach space with an unconditional basis, we exemplify our approach for the spaces $c_0$ and $\ell^1$. Thereafter we will apply our abstract construction to give counterexamples on the $\ell^p$-spaces. For this one uses a non-standard representation of these spaces which is well-known in the geometry of Banach spaces. After embedding these counterexamples into $L^p(\IR)$, we can give a counterexample to the extrapolation problem for maximal regularity. In the last section we show that the same techniques can be used to construct systematically bases which are not $\mathcal{R}$-bases.

\section{Schauder multipliers as generators of holomorphic semigroups}

For the rest of the paper we will only consider infinite-dimensional and complex Banach spaces. In this section we give the necessary definitions from the theory of Schauder bases and show how to construct semigroups with the help of Schauder bases.  

\begin{definition}[Schauder basis] A sequence $(e_m)_{m \in \IN}$ in a Banach space $X$ is called a \emph{Schauder basis} if for each $x \in X$ there is a unique sequence of scalars $(a_m)_{m \in \IN}$ such that
	\begin{equation*} x = \sum_{m=1}^{\infty} a_m e_m. \end{equation*}
A sequence $(e_m)_{m \in \IN}$ is called a \emph{basic sequence} if it is a basis for the closed linear span of $(e_m)_{m \in \IN}$.
The functional $e_m^{*} \in X^*$ which maps $x$ to the unique $m$-th coefficient in the expansion of $x$ is called the \emph{$m$-th coordinate functional}.
\end{definition}

The mere concept of a Schauder basis is sometimes too general to be useful in practice, therefore one often considers special bases with additional properties. The most important example is that of an unconditional basis.

\begin{definition}[Unconditional basis] A Schauder basis $(e_m)_{m \in \IN}$ for a Banach space $X$ is called \emph{unconditional} if for each $x \in X$ the unique expansion $x = \sum_{m=1}^{\infty} a_m e_m$ converges unconditionally, i.e. $\sum_{m=1}^{\infty} a_{\pi(m)} e_{\pi(m)} = x$ for each permutation $\pi$ of the natural numbers.
\end{definition}

For unconditional bases one can show that the convergence of $\sum_{m=1}^{\infty} a_m e_m$ implies the convergence of $\sum_{m=1}^{\infty} b_m a_m e_m$ for all $(b_m)_{m \in \IN} \in \ell^{\infty}$. The closed graph theorem and the uniform boundedness principle show that there exists a smallest constant $K \ge 0$ such that
	\begin{equation*} \biggnorm{\sum_{m=1}^{\infty} b_m a_m e_m} \le K \biggnorm{\sum_{m=1}^{\infty} a_m e_m} \end{equation*}
holds for all $(a_m)_{m \in \IN}$ for which the expansion converges and all choices of sequences $(b_m)_{m \in \IN}$ with $\norm{(b_m)}_{\infty} \le 1$. This constant is called the \emph{unconditional constant} of $(e_m)_{m \in \IN}$.

It is now time to link Schauder bases with semigroup theory. 

\begin{definition}[Schauder Multiplier]\label{definition:schauder_multiplicator} Let $(e_m)_{m \in \IN}$ be a Schauder basis for a Banach space $X$. For a complex sequence $(\gamma_m)_{m \in \IN}$ the closed linear \emph{Schauder multiplier $A$ associated} to $(\gamma_m)$ is defined by
	\begin{align*}
		D(A) = \biggl\{ x = \sum_{m=1}^{\infty} a_m e_m: \sum_{m=1}^{\infty} \gamma_m a_m e_m \text{ exists} \biggr\} \\
		A \biggl( \sum_{m=1}^{\infty} a_m e_m \biggr) = \sum_{m=1}^{\infty} \gamma_m a_m e_m. 
	\end{align*}
\end{definition}

A nice feature of Schauder multipliers is that they can be used to construct systematically generators of semigroups. The proof of the following theorem can be found in~\cite[Theorem~3.2]{Ven93}.

\begin{theorem}\label{theorem:semigroup_generation} Let $(e_m)_{m \in \IN}$ be a basis for a Banach space $X$. Let $(\gamma_m)_{m \in \IN}$ be an increasing sequence of positive real numbers and let $A$ be the multiplier associated to $(\gamma_m)$. Then $-A$ generates a holomorphic $C_0$-semigroup $(T(z))_{z \in \Sigma_{\pi/2}}$ and $T(z)$ is the multiplier associated to the sequence $(e^{-z\gamma_m})_{m \in \IN}$. 
\end{theorem}

The situation is much simpler if the basis $(e_m)_{m \in \IN}$ is unconditional. In this case a Schauder multiplier is bounded if and only if the associated sequence is in $\ell^{\infty}$. Hence, the multiplier associated to a sequence $(-\gamma_m)_{m \in \IN}$ generates a holomorphic semigroup if $\gamma_m \in \Sigma_{\delta}$ for all $m \in \IN$ for some $\delta < \frac{\pi}{2}$. In this case one always has:

\begin{theorem}\label{thm:unconditionality_implies_maxreg}
	Let $(e_m)_{m \in \IN}$ be an unconditional basis for some UMD-space $X$ and let $A$ be the multiplier associated to some sequence $(\gamma_m)_{m \in \IN}$ with $\gamma_m \in \Sigma_{\delta}$ for all $m \in \IN$ for some $\delta < \frac{\pi}{2}$. Then $-A$ has maximal regularity. 
\end{theorem}
\begin{proof}
	We show that $A$ has a bounded $H^{\infty}(\Sigma_{\phi})$-calculus for $\frac{\pi}{2} > \phi > \delta$. This implies maximal regularity by a result of Kalton and Weis~\cite[Theorem~5.3]{KalWei01} (notice that UMD-spaces have property $(\Delta)$). Let $f \in H_0^{\infty}(\Sigma_{\phi})$. Notice that $f(A)$ is the multiplier associated to the sequence $(f(\gamma_m))_{m \in \IN}$. If $K$ denotes the unconditional constant of $(e_m)_{m \in \IN}$, we have
		\[ \biggnorm{\sum_{m=1}^{\infty} f(\gamma_m) a_m e_m} \le K \norm{f}_{\infty} \biggnorm{\sum_{m=1}^{\infty} a_m e_m}, \]
	which is exactly the boundedness of the $H^{\infty}(\Sigma_{\phi})$-calculus of $A$. 
\end{proof}

\begin{remark}
	Conversely, for a conditional basis the multiplier associated to the sequence $(2^m)_{m \in \IN}$ never has a bounded $H^{\infty}$-calculus as the sequence is interpolating. We refer to~\cite[Example, p.~29]{Are04} for details. Notice that this is enough to show that having a bounded $H^{\infty}$-calculus does not extrapolate: the trigonometric basis $(e^{imx})_{m \in \IZ}$ on $L^p([0,2\pi])$ for $p \in (1, \infty)$ is unconditional if and only if $p=2$. Hence, the consistent semigroups generated by
		\[ -A_p\biggl(\sum_{m=1}^{\infty} a_m e^{imx}\biggr) = -\sum_{m=1}^{\infty} 2^m a_m e^{imx} \]
	have a bounded $H^{\infty}$-calculus if and only if $p = 2$. Nevertheless, these generators have maximal regularity~\cite[Example, p.~31]{Are04} (notice that BIP implies maximal regularity on UMD-spaces). There is also a positive result: as in the case of maximal regularity the boundedness of the $H^{\infty}$-calculus extrapolates to $L^p$ if the semigroup satisfies Gaussian bounds~\cite[Theorem~3.1]{DuoRob96}.
\end{remark}

\section{Transference to \texorpdfstring{$\Rad(X)$}{Rad(X)}}

In this section we explain how we can describe $\mathcal{R}$-analyticity of a semigroup in terms of holomorphy of an associated semigroup on $\Rad(X)$. In our counterexamples we will always show that the holomorphy of this semigroup is violated instead of working directly with $\mathcal{R}$-analyticity.

\begin{definition}[Associated Semigroup on $\Rad(X)$]\label{definition:associated_semigroup} Let $(T(z))_{z \in \Sigma}$ be a holomorphic $C_0$-semigroup on a Banach space $X$. Given a sequence $(q_n)_{n \in \IN} \subset (0,1)$, one defines the \emph{associated semigroup} $(\mathcal{T}(z))_{z \in \Sigma}$ on the finite Rademacher sums by
	\begin{equation*}
		\mathcal{T}(z) \biggl( \sum_{n=1}^{N} r_n x_n \biggr) \coloneqq \sum_{n=1}^{N} r_n T(q_n z) x_n \qquad (z \in \Sigma, N \in \IN).
	\end{equation*}
\end{definition}

For $x \in X$ one often uses the notation $r_n \otimes x$ for the function $\omega \mapsto r_n(\omega) x$ in $\Rad(X)$.

\begin{remark} Let $(T(z))_{z \in \Sigma}$ be the holomorphic $C_0$-semigroup generated by the multiplier $-A$ associated to a sequence $(-\gamma_m)_{m \in \IN}$ as in Theorem~\ref{theorem:semigroup_generation}. Then for $\mathcal{T}(z)$ we have
	\begin{equation}
		\label{eq:rad_expression}
		\begin{split} 
			\mathcal{T}(z) \biggl(\sum_{m,n=1}^N a_{nm} r_n \otimes e_m \biggr) & = \sum_{n=1}^N r_n T(q_n z)\biggl( \sum_{m=1}^N a_{nm} e_m \biggr) \\
			& = \sum_{n,m=1}^N e^{- q_n \gamma_m z} a_{nm} r_n \otimes e_m.
		\end{split}
	\end{equation}
\end{remark}

We will only show part~\ref{theorem:arendt_bu} of the following transference theorem because the converse is not necessary for constructing counterexamples (for a proof of part~\ref{theorem:arendt_bu_converse} see~\cite[Theorem~3.6]{AreBu03}).

\begin{theorem}[W. Arendt, S. Bu]\label{theorem:arendt_bu_all} Let $(T(z))$ be a holomorphic $C_0$-semigroup on a Banach space $X$. Then the following hold:
	\begin{thm_enum}
		\item\label{theorem:arendt_bu} If $(T(z))$ is $\mathcal{R}$-analytic, then the associated semigroup extends to a holomorphic $C_0$-semigroup $(\mathcal{T}(z))$ on $\Rad(X)$. Moreover, there exist $M, \omega > 0$ such that
			\begin{equation*} \norm{\mathcal{T}(z)} \le M e^{\omega \abs{z}} \end{equation*}
		holds independently of the chosen sequence $(q_n)_{n \in \IN} \subset (0,1)$.
		\item\label{theorem:arendt_bu_converse} Conversely, if the associated semigroup $(\mathcal{T}(z))$ is strongly continuous and holomorphic for some $(q_n)_{n \in \IN}$ being dense in $(0,1)$, then $(T(z))$ is $\mathcal{R}$-analytic. 
	\end{thm_enum}
\end{theorem}
\begin{proof}
	For $z$ in some sufficiently small sector $\Sigma$ one has
	\begin{align*}
		\MoveEqLeft \biggnorm{\mathcal{T}(z)\biggl(\sum_{n=1}^N r_n x_n \biggr)} = \biggnorm{\sum_{n=1}^N r_n T(q_n z) x_n} \le \R{T(\lambda z): \lambda \in (0,1)} \biggnorm{\sum_{n=1}^N r_n x_n}.
	\end{align*} 
	Since the finite Rademacher sums are dense in $\Rad(X)$, $\mathcal{T}(z)$ extends to a bounded linear  operator on $\Rad(X)$. Now let $z \in \Sigma$ be arbitrary and $M$ be given by $M \coloneqq \R{T(z): z \in \Sigma, \abs{z} \le 1}$. There exist unique $n \in \IN$, $s \in [0,1)$ such that $z = (n + s)\frac{z}{\abs{z}}$. Then for $\omega = \log M$ one has 
		\begin{align*} 
			\norm{\mathcal{T}(z)} \le \biggnorm{\mathcal{T}\biggl( \frac{z}{\abs{z}} \biggr)}^n \biggnorm{\mathcal{T} \biggl( s\frac{z}{\abs{z}} \biggr)} \le M e^{n \log M} \le M e^{\omega \abs{z}}.
		\end{align*}
	The strong continuity can easily be checked for finite Rademacher sums and can be extended to arbitrary elements of $\Rad(X)$ by the local boundedness of $z \mapsto \mathcal{T}(z)$ in operator norm.
\end{proof}

\section{Warm-up: Counterexamples on \texorpdfstring{$c_0$}{c\_0} and \texorpdfstring{$\ell^1$}{l\_1}}

Before we consider counterexamples on general Banach spaces with an unconditional basis, we construct counterexamples for the concrete Banach spaces $c_0$ and $\ell^1$. They also illustrate our approach. The following elementary lemma will be useful in the future and throws light on the special role played by the sequence $(2^m)_{m \in \IN}$ when used as multiplier sequence.

\begin{lemma}\label{lemma:max_difference} The function $d(t) = e^{-2^m t} - e^{-2^{m+1}t}$ $(m \in \IN)$ possesses a unique maximum in $[0,1]$ at $t_0 = \frac{\log 2}{2^m}$. Moreover, the maximum value $d(t_0) = \frac{1}{4}$ is independent of $m$.
\end{lemma}

\subsection{The space \texorpdfstring{$c_0$}{c\_0}}

Let $(e_m)_{m \in \IN}$ be the standard unit vector basis of $c_0$. Then the \emph{summing basis} $(s_m)_{m \in \IN}$ given by $s_m \coloneqq \sum_{k=1}^m e_k$ is a conditional basis of $c_0$~\cite[Example~3.1.2]{AlbKal06}.

\begin{proposition}\label{prop:ce_c_0}
	Let $(s_m)_{m \in \IN}$ be the summing basis of $c_0$. Then $-A$ given by
	\begin{align*}
		D(A) = \biggl\{ x = \sum_{m=1}^{\infty} a_m s_m: \sum_{m=1}^{\infty} 2^m a_m s_m \quad \mathrm{ exists} \biggr\} \\
		A \biggl( \sum_{m=1}^{\infty} a_m s_m \biggr) = \sum_{m=1}^{\infty} 2^m a_m s_m
	\end{align*}
	generates a holomorphic $C_0$-semigroup $(T(z))_{z \in \Sigma_{\pi/2}}$ that is not $\mathcal{R}$-bounded on $[0,1]$.
\end{proposition}
\begin{proof}
	Assume that $\R{T(t): 0 < t \le 1} < \infty$. Then $(\mathcal{T}(t))_{t \ge 0}$ is a $C_0$-semigroup on $\Rad(c_0)$. We now consider $x_N \coloneqq \sum_{m=1}^N (s_{2m} - s_{2m-1}) \otimes r_m$ for $N \in \IN$. Its norm in $\Rad(c_0)$ is
	\begin{equation*} \biggnorm{\sum_{m=1}^N (s_{2m} - s_{2m-1}) \otimes r_m} = \int_0^1 \biggnorm{\sum_{m=1}^N r_m(\omega) e_{2m}}_{\infty} d\omega = 1. \end{equation*}
	One has by~\eqref{eq:rad_expression}
	\begin{equation*} \mathcal{T}(1)(x_N) = \sum_{m=1}^N (e^{-2^{2m} q_m} s_{2m} - e^{-2^{2m-1} q_m} s_{2m-1}) \otimes r_m. \end{equation*}
	In particular for the choice $q_m = \frac{\log 2}{2^{2m-1}}$ the first coordinate of the above expression in $\omega$ is given by $-\frac{1}{4} \sum_{m=1}^N r_m(\omega)$. Contradictory to our assumption that $\mathcal{T}(1)$ is bounded this shows that
	\begin{equation*} \norm{\mathcal{T}(1)(x_N)} \ge \frac{1}{4} \int_0^1 \biggabs{\sum_{m=1}^N r_m(\omega)} d\omega \ge \frac{C^{-1}}{4} \sqrt{N} \end{equation*}
	by Khintchine's inequality. So $(T(t))_{t \ge 0}$ is not $\mathcal{R}$-bounded on $[0,1]$.
\end{proof}

\subsection{The space \texorpdfstring{$\ell^1$}{l\_1}}

Let $(e_m)_{m \in \IN}$ denote the standard unit vector basis of $\ell^1$. Then $(f_m)_{m \in \IN}$ given by $f_1 = e_1$ and $f_m = e_m - e_{m-1}$ for $m \ge 2$ is a conditional basis for $\ell^1$~\cite[Example~14.2]{Sin70}. Again this conditional basis can be used to construct a counterexample.

\begin{proposition}\label{prop:ce_ell_1} 
	Let $(f_m)_{m \in \IN}$ be the basis of $\ell^1$ defined above. Then $-A$ given by
	\begin{align*}
		D(A) = \biggl\{ x = \sum_{m=1}^{\infty} a_m f_m: \sum_{m=1}^{\infty} 2^m a_m f_m \quad \mathrm{ exists} \biggr\} \\
		A \biggl( \sum_{m=1}^{\infty} a_m f_m \biggr) = \sum_{m=1}^{\infty} 2^m a_m f_m
	\end{align*}
	generates a holomorphic $C_0$-semigroup $(T(z))_{z \in \Sigma_{\pi/2}}$ that is not $\mathcal{R}$-bounded on $[0,1]$.
\end{proposition}

\begin{proof}
	One proceeds as in the proof of Proposition~\ref{prop:ce_c_0}. This time one looks at the vector $x_N = \sum_{n,m=1}^N r_n \otimes f_m = \sum_{n=1}^N r_n \otimes e_N$ in $\Rad(\ell_1)$ for $N \in \IN$. By Khintchine's inequality, its norm in $\Rad(\ell^1)$ is
		\begin{equation*} \biggnorm{\sum_{n,m=1}^N r_n \otimes f_m} = \int_0^1 \biggnorm{\sum_{n=1}^N r_n(\omega) e_N}_{\ell^1} d\omega = \int_0^1 \biggabs{\sum_{n=1}^N r_n(\omega)} d\omega \le C \sqrt{N}. \end{equation*}
	A short calculation using~\eqref{eq:rad_expression} shows that $\mathcal{T}(1)x_N$ is given by
		\begin{equation*}
			\sum_{n,m=1}^N e^{-2^m q_n} r_n \otimes f_m = \sum_{n=1}^N e^{-2^N q_n} r_n \otimes e_N + \sum_{m=1}^{N-1} \sum_{n=1}^N (e^{-2^m q_n} - e^{-2^{m+1} q_n}) r_n\otimes e_m
		\end{equation*}
	Now, a second application of Khintchine's inequality shows that	 one has
	\begin{align*}
			\norm{\mathcal{T}(1)x_N} & \ge \sum_{m=1}^{N-1} \int_0^1 \biggabs{\sum_{n=1}^N (e^{-2^m q_n} - e^{-2^{m+1} q_n}) r_n(\omega)} d\omega \\
			& \ge C^{-1} \sum_{m=1}^{N-1} \biggl( \sum_{n=1}^N \abs{e^{-2^m q_n} - e^{-2^{m+1} q_n}}^2 \biggr)^{1/2}.
		\end{align*}
	We again choose $q_m = \frac{\log 2}{2^m}$. By estimating the right hand side, we obtain
		\begin{equation*} \norm{\mathcal{T}(1)x_N} \ge C^{-1} \sum_{m=1}^{N-1} \frac{1}{4} = \frac{N-1}{4C}. \end{equation*}
	As in the last proposition, this shows that $(T(t))_{t \ge 0}$ is not $\mathcal{R}$-bounded on $[0,1]$. 
\end{proof}

\section{A new proof of the Kalton-Lancien Theorem}\label{sec:general_results}

We now consider Banach spaces with an unconditional basis. The following notions from the geometry of Banach spaces will be useful.

\begin{definition}[Equivalence of Bases] Two basic sequences $(e_m)_{m \in \IN}$ and $(f_m)_{m \in \IN}$ for two Banach spaces $X$ and $Y$ are called \emph{equivalent} if for each sequence $(a_m)_{m \in \IN}$ of complex numbers
	\begin{equation*} \sum_{m=1}^{\infty} a_m e_m \quad \text{converges in $X$ if and only if} \quad \sum_{m=1}^{\infty} a_m f_m \quad \text{converges in $Y$.} \end{equation*}
\end{definition}

Notice that if $(e_m)_{m \in \IN}$ is an unconditional basis for a Banach space $X$, then $(e_{\pi(m)})_{m \in \IN}$ is an unconditional basis for $X$ for all permutations $\pi: \IN \to \IN$.

\begin{definition}[Symmetric Basis] An unconditional basis $(e_m)_{m \in \IN}$ for a Banach space $X$ is \emph{symmetric} if $(e_m)_{m \in \IN}$ is equivalent to $(e_{\pi(m)})_{m \in \IN}$ for all permutations $\pi$ of $\IN$.
\end{definition}

\begin{proposition}\label{prop:ce_constructor}
	Let $X$ be a Banach space with an unconditional, non-symmetric normalized Schauder basis $(e_m)_{m \in \IN}$. Then there exists a generator $-A$ of a holomorphic $C_0$-semigroup $(T(z))_{z \in \Sigma_{\pi/2}}$ on $X$ which is not $\mathcal{R}$-bounded on $[0,1]$. More precisely, there exists a Schauder basis $(f_m)_{m \in \IN}$ of $X$ such that $A$ is given by
	\begin{align*}
		D(A) = \biggl\{ x = \sum_{m=1}^{\infty} a_m f_m: \sum_{m=1}^{\infty} 2^m a_m f_m \quad \mathrm{ exists} \biggr\} \\
		A \biggl( \sum_{m=1}^{\infty} a_m f_m \biggr) = \sum_{m=1}^{\infty} 2^m a_m f_m.
	\end{align*}
\end{proposition}
\begin{proof}
	As in the first part of the proof of~\cite[Proposition~23.2]{Sin70} one deduces from the fact that $(e_m)_{m \in \IN}$ is non-symmetric that there exists a permutation $\pi$ of the even numbers such that the unconditional basic sequences $(e_{2m - 1})_{m \in \IN}$ and  $(e_{\pi(2m)})_{m \in \IN}$ are not equivalent. Exactly as in the proof, both
	\begin{align*}
		f_m^{\prime} = \begin{cases} e_m & \text{m odd} \\ e_{\pi(m)} + e_{m-1} & \text{m even} \end{cases} \quad \text{and} \quad f_m^{\prime\prime} = \begin{cases} e_m + e_{\pi(m+1)} & \text{m odd} \\ e_{\pi(m)} & \text{m even} \end{cases}
	\end{align*}
	are Schauder bases for $X$. Since $(e_{2m - 1})_{m \in \IN}$ and  $(e_{\pi(2m)})_{m \in \IN}$ are not equivalent, there exists a sequence $(a_m)_{m \in \IN}$ such that the expansion with respect to the coefficients $(a_m)_{m \in \IN}$ converges for $(e_{2m - 1})_{m \in \IN}$ or for $(e_{\pi(2m)})_{m \in \IN}$ but not for both. For the rest of the proof we will assume without loss of generality that the expansion converges for $(e_{\pi(2m)})_{m \in \IN}$ (if not simply replace $f_m^{\prime}$ by $f_m^{\prime\prime}$ in the next steps). Let $f_m \coloneqq f_m^{\prime}$. We now define $A$ as in the statement of the proposition. By Theorem~\ref{theorem:semigroup_generation}, $-A$ generates a holomorphic semigroup $(T(z))_{z \in \Sigma_{\pi/2}}$. Assume that $(T(t))_{t \ge 0}$ is $\mathcal{R}$-bounded on $[0,1]$. Then for each choice of $(q_n)_{n \in \IN} \subset (0,1)$ the associated semigroup $(\mathcal{T}(t))_{t \ge 0}$ - as defined in Definition~\ref{definition:associated_semigroup} - is a (holomorphic) $C_0$-semigroup on $\Rad(X)$ by Theorem~\ref{theorem:arendt_bu_all}\ref{theorem:arendt_bu}. We now show that
		\begin{equation} \sum_{m=1}^{\infty} a_m e_{\pi(2m)} \otimes r_m \label{eq:ce_argument} \end{equation}
	converges in $\Rad(X)$. Indeed, for fixed $\omega \in [0,1]$ the series $\sum_{m=1}^{\infty} a_m r_m(\omega) e_{\pi(2m)}$ converges by the unconditionality of $(e_{\pi(2m)})_{m \in \IN}$ as $r_m(\omega) \in \{-1, 1\}$ for every $m \in \IN$. Hence, the above series defines a measurable function as the pointwise limit of measurable functions. Moreover, if $K$ denotes the unconditional constant of $(e_{\pi(2m)})_{m \in \IN}$, one has for each $\omega \in [0,1]$
		\begin{equation} \biggnorm{\sum_{m=1}^{\infty} r_m(\omega) a_m e_{\pi(2m)}} \le K \biggnorm{\sum_{m=1}^{\infty} a_m e_{\pi(2m)}}. \label{eq:ce_l1_convergence} \end{equation}
	This shows that the series \eqref{eq:ce_argument} is in $L^1([0,1]; X)$. Using an analogous estimate as \eqref{eq:ce_l1_convergence}, one sees that the sequence of partial sums $\sum_{m=1}^{N} a_m e_{\pi(2m)} \otimes r_m$ converges to $\sum_{m=1}^{\infty} a_m e_{\pi(2m)} \otimes r_m$ in $\Rad(X)$. Notice that $e_{\pi(2m)} = f_{2m} - f_{2m-1}$. By the continuity of $\mathcal{T}(1)$, we obtain from~\eqref{eq:rad_expression} that
		\begin{align*}
			\MoveEqLeft h \coloneqq \mathcal{T}(1)\biggl(\sum_{m=1}^{\infty} a_m (f_{2m} - f_{2m-1}) \otimes r_m \biggr) \\
			& = \lim_{N \to \infty} \sum_{m=1}^{N} a_m \left( e^{-2^{2m} q_m} f_{2m} - e^{-2^{2m-1} q_m} f_{2m-1} \right) \otimes r_m \\
			& = \lim_{N \to \infty} \sum_{m=1}^{N} e^{-2^{2m} q_m} a_m e_{\pi(2m)} \otimes r_m + a_m (e^{-2^{2m} q_m} - e^{-2^{2m-1} q_m}) e_{2m-1} \otimes r_m
		\end{align*}
	exists in $\Rad(X)$. Now choose $q_m = \frac{\log 2}{2^{2m-1}}$ as discussed in Lemma~\ref{lemma:max_difference}. Then after choosing a subsequence $(N_k)$ there exists a set $N \subset [0,1]$ of measure zero such that
		\begin{align}
			\frac{1}{4}  \sum_{m=1}^{N_k} \left( a_m r_m(\omega) e_{\pi(2m)} - a_m r_m(\omega) e_{2m-1} \right) \xrightarrow[k \to \infty]{} h(\omega) \quad \text{for all } \omega \in N^C. \label{eq:ce_after_riesz}
		\end{align}
	Applying the coordinate functionals for $(e_m)_{m \in \IN}$ to \eqref{eq:ce_after_riesz} shows that for $\omega \in N^C$ the unique coefficients $(e^*_m(h(\omega)))$ of the expansion of $h(\omega)$ with respect to $(e_m)_{m \in \IN}$ satisfy $e^*_{2m-1}(h(\omega)) = -\frac{a_m}{4} r_m(\omega)$. Since $(e_m)_{m \in \IN}$ is unconditional,
	\begin{align*}
		\sum_{m=1}^{\infty} a_m r_m(\omega) e_{2m-1} \quad \text{and therefore} \quad \sum_{m=1}^{\infty} a_m e_{2m-1} \quad \text{converge.}
	\end{align*}
	This contradicts our assumptions and $(T(t))_{t \ge 0}$ cannot be $\mathcal{R}$-bounded on $[0,1]$.
\end{proof}

\begin{remark} Notice that in the spaces $L^p([0,1])$ $(p \in (1, \infty) \setminus \{2\})$ every unconditional basis is non-symmetric~\cite[Theorem~21.1]{Sin70}. So the above construction in particular works for the normalized Haar basis.
\end{remark}

The next proposition originally due to J. Lindenstrauss and M. Zippin~\cite[Note (1) at the end]{LinZip69} shows that the above result is applicable to almost all Banach spaces with an unconditional basis. We only present a sketch of the proof from~\cite[Theorem~7.6]{McA72} to show a link to the next section.

\begin{proposition}[Existence of non-symmetric Bases]\label{prop:non-symmetric} Let $X$ be a Banach space with an unconditional basis. If $X$ is not isomorphic to $c_0, \ell^1$ or $\ell^2$, then $X$ has a normalized unconditional, non-symmetric basis.
\end{proposition}
\begin{proof}
	Assume that every normalized unconditional Schauder basis for $X$ is symmetric. Let $(e_m)_{m \in \IN}$ be a normalized unconditional, hence symmetric, basis for $X$. Then it can be shown that $(e_m)_{m \in \IN}$ is equivalent to all of its normalized block basic sequences, i.e. $(e_m)_{m \in \IN}$ is perfectly homogeneous. By a famous theorem of M.~Zippin~\cite[Theorem~9.1.8]{AlbKal06}, $X$ is isomorphic to $c_0$ or $\ell^p$ for $1 \le p < \infty$. However, we will see in Section~\ref{section:l_p} the well-known result that for $p \in (1, \infty) \setminus \{2\}$ the spaces $\ell^p$ have a non-symmetric, unconditional Schauder basis.
\end{proof}

\begin{remark} Conversely, all normalized unconditional bases in the spaces $c_0$, $\ell^1$ and $\ell^2$ are symmetric for the simple reason that all normalized unconditional bases in these spaces are unique up to equivalence~\cite[Theorem~2.b.10]{LinTza77}.
\end{remark}

Now, the main result of this section is an easy consequence.

\begin{theorem}[Kalton-Lancien Revisited] Let $X$ be a Banach space with an unconditional basis. Assume that $X$ has (MRP). Then $X \simeq \ell^2$. More precisely, for $X \not\simeq \ell^2$ there exists a Schauder basis $(f_m)_{m \in \IN}$ for $X$ such that $-A$ given by
	\begin{align*}
		D(A) = \biggl\{ x = \sum_{m=1}^{\infty} a_m f_m: \sum_{m=1}^{\infty} 2^m a_m f_m \quad \mathrm{ exists} \biggr\} \\
		A \biggl( \sum_{m=1}^{\infty} a_m f_m \biggr) = \sum_{m=1}^{\infty} 2^m a_m f_m
	\end{align*}
	generates a holomorphic $C_0$-semigroup $(T(z))_{z \in \Sigma_{\pi/2}}$ on $X$ that is not $\mathcal{R}$-bounded on $[0,1]$ and in particular does not have maximal regularity.
\end{theorem}
\begin{proof}
	If $X$ is not isomorphic to $c_0$ or $\ell^1$, Proposition~\ref{prop:non-symmetric} shows that we can apply Proposition~\ref{prop:ce_constructor} which yields the desired counterexample. In the cases of $X \simeq c_0$ or $X \simeq \ell^1$ we showed the theorem by hand in Proposition~\ref{prop:ce_c_0} and Proposition~\ref{prop:ce_ell_1}. 
\end{proof}

\section{Counterexamples on \texorpdfstring{$L^p$}{L\_p}-spaces}\label{section:l_p}

In this section we want to show that the property of having maximal regularity does not extrapolate. In particular, we construct consistent semigroups $(T_p(t))_{t \ge 0}$ on $L^p$ for $p \in (1, \infty)$ with the following properties: $(T_2(t))_{t \ge 0}$ and therefore all $(T_p(t))_{t \ge 0}$ are holomorphic but $(T_p(t))_{t \ge 0}$ has maximal regularity if and only if $p = 2$. This is done by using a concrete non-symmetric basis for $\ell^p$ for which the general results of Section~\ref{sec:general_results} apply. The basis is obtained from a non-standard representation of the space $\ell^p$ using the following construction.

\begin{definition} Let $(X_n)_{n \in \IN}$ be a sequence of Banach spaces. Then for $p \in [1, \infty)$ the $\ell^p$-sum $\oplus_{\ell^p}^n X_n$ of $(X_n)_{n \in \IN}$ is given by	
	\[ \oplus_{\ell^p}^n X_n \coloneqq \biggl\{ (x_n) \in \prod_{n=1}^{\infty} X_n : \sum_{n=1}^{\infty} \norm{x_n}_{X_n}^p < \infty \biggr\}, \qquad \norm{(x_n)} \coloneqq \biggl( \sum_{n=1}^{\infty} \norm{x_n}_{X_n}^p \biggr)^{1/p}. \]
\end{definition}

	We will be interested in the spaces $X^p \coloneqq \oplus_{\ell^p}^n \ell^2_n$. Observe that the standard unit vector basis of $X^p$ seen as a sequence space is normalized and unconditional. For $p \in (1, \infty) \setminus \{2\}$ it is not equivalent to the standard basis of $\ell^p$. Indeed, for $1 < p < 2$ consider the sequence $x_k = \sqrt[p]{\frac{1}{2^n}}$ for $k = \frac{(2^n -1)2^n}{2} + 1, \ldots, \frac{2^n(2^n + 1)}{2}$ $(n \in \IN)$ and $x_k = 0$ in any other case. Then one has $\sum_{k=1}^{\infty} \abs{x_k}^p = \sum_{n=1}^{\infty} 2^n \frac{1}{2^n} = \infty$, but 
		\[ \norm{(x_k)}^p_{X^p} = \sum_{n=1}^{\infty} (2^{n(1-\frac{2}{p})})^{p/2} = \sum_{n=1}^{\infty} 2^{n(\frac{p}{2}-1)} < \infty \quad \text{as } p < 2. \]
	In the case $2 < p < \infty$ let $(x_k)_{k \in \IN}$ be the sequence obtained by inserting the sequence $(\frac{1}{\sqrt{k}})_{k \in \IN}$ into the set $\cup_{n=1}^{\infty} [\frac{(2^n -1)2^n}{2} + 1, \frac{2^n(2^n + 1)}{2}]$. Then $\sum_{k=1}^{\infty} \abs{x_k}^p = \sum_{k=1}^{\infty} k^{-p/2} < \infty$ as $p > 2$. However, one has $\sum_{k=\frac{(2^n-1)2^n}{2} +1}^{\frac{2^n(2^n+1)}{2}} \abs{x_k}^2 \ge \frac{1}{2}$ by the well-known argument for the divergence of the harmonic series. Hence, $\norm{(x_k)}_{X^p} = \infty$. 
	
	Moreover, using Pełczyński's decomposition technique, one can show that for fixed $p \in (1, \infty)$ one has $X^p \simeq \ell^p$. Since the standard basis $(e_m)_{m \in \IN}$ of $X^p$ is not equivalent to the standard basis of $\ell^p$ (for $p \in (1, \infty) \setminus \{2\})$, the general theory shows that $(e_m)_{m \in \IN}$ cannot be symmetric~\cite[Proposition~21.5]{Sin70}. More easily, choose $\pi(2m) = \frac{(m - 1)m}{2} + 1$ and use successively $\pi(2m+1)$ to fill up the rest. Then $[e_{\pi(2m)}]_{m \in \IN}$ is isometrically isomorphic to $\ell^p$ and versions of the counterexamples above show that $(e_{\pi(2m)})_{m \in \IN}$ is not equivalent to $(e_{2m})_{m \in \IN}$.
	Note that one sees directly that $X^2$ is isometrically isomorphic to $\ell^2$ and that the standard basis of $X^2$ is equivalent to the standard Hilbert space basis of $\ell^2$. 
	
One can now use Proposition~\ref{prop:ce_constructor} to obtain a counterexample on $X^p \simeq \ell^p$ for $p \in (1, \infty) \setminus \{2\}$. However, we want to do more: we want to define a consistent family of counterexamples on $X^p$ in the scale $p \in (1, \infty) \setminus \{2\}$. For this it is necessary to find explicitly $p$-independent choices of both the permutations and the bases $f_m$ used in the proof of Proposition~\ref{prop:ce_constructor}.

\begin{proposition}\label{prop:concrete_bases} Let $(e_m)_{m \in \IN}$ be the standard unit vector basis of $X^p$ ($p \in (1, \infty))$. Then there exists a $p$-independent permutation $\pi$ of the even numbers such that the choice $f_m^{\prime}$ yields semigroups without maximal regularity for $p \in (2, \infty)$.
\end{proposition}
\begin{proof}
	The permutation $\pi$ of the even numbers is defined as follows. Let $b_0, b_1, b_2, \ldots$ be the first even numbers in the blocks $B_k \coloneqq [ \frac{(k-1)k}{2} + 1, \frac{k(k+1)}{2} ]$ $(k \in \IN)$, so $b_0 = 2$, $b_1 = 4$, $b_2 = 8$, $b_3 = 12$, $b_4 = 16$, $b_5 = 22$, $b_6 = 30$ and so on. Now, we define
		\[ \pi(m) = \begin{cases}
				m & m \text{ odd} \\
				b_k & m = 4k + 2 \\
				\min 2\IN \setminus (\{ b_n: n \in \IN \} \cup \pi ([1, m-1])) & m = 4k.
			 \end{cases}
		\]
	 The permutation $\pi$ jumps to the first even number of some block $B_k$ in every second permutation step of the even numbers and collects all other even numbers in the other steps. Notice that a sequence of the form $(a_1, 0, a_2, 0, a_3, 0, \ldots)$ converges with respect to $(e_{\pi(2m)})_{m \in \IN}$ if and only if $(a_m)_{m \in \IN} \in \ell^p$. This observation together with a slight modification of the above counterexamples shows that $(e_{\pi(2m)})_{m \in \IN}$ is not equivalent to $(e_{2m+1})_{m \in \IN}$.
	  
	 Moreover, the above arguments show that for $p \in (2, \infty)$ there exists a sequence $(a_m)_{m \in \IN}$ which converges with respect to $(e_{\pi(2m)})_{m \in \IN}$ but not with respect to $(e_{2m+1})_{m \in \IN}$. Thus in the case $p \in (2, \infty)$ one can use $(f_m^{\prime})_{m \in \IN}$ to construct a counterexample.
\end{proof}

The above arguments leave open what happens in the case $p \in (1,2)$. As we can construct a counterexample to the extrapolation problem without addressing this issue, we will postpone the discussion of this question to the end of the section.

Notice that we have not yet found a counterexample to the extrapolation problem although we have found consistent semigroups on $X^p$ with the desired properties. For this $X^p \simeq \ell^p$ is not sufficient because we also need the consistency of the isomorphisms for different $p$. Sadly, the usual argument using Pełczyński's decomposition technique does not seem to yield such consistent isomorphisms. In a different direction one could try to apply~Proposition~\ref{prop:ce_constructor} to the normalized Haar basis of $L^p$. This works perfectly for a fixed $p \in (1,\infty)$, but the Haar basis cannot be simultaneously normalized for all or two different choices of $p$. This issue was overlooked in the presentation given by the author in~\cite{Fac13}. 

However, there is a way to embed the above family of counterexamples on $X^p$ consistently into a scale of $L^p$-spaces.

\begin{theorem}\label{thm:l^p_counterexample}
	There exist consistent holomorphic $C_0$-semigroups $(T_p(z))_{z \in \Sigma_{\pi/2}}$ on $L^p((0, \infty))$ for $p \in (1, \infty)$ such that $(T_p(z))_{z \in \Sigma_{\pi/2}}$ does not have maximal regularity for $p \in (2,\infty)$. 
\end{theorem}
\begin{proof}
	Let $r_1, r_2, r_3, \ldots$ be the Rademacher functions. Notice that for each $n$ one has an isomorphism 
	\begin{align*}
		\ell^2_n \simeq \linspan \{r_1, \ldots, r_n \} \eqqcolon \Rad_n \qquad (a_1, \ldots, a_n) \mapsto \sum_{k=1}^n a_k r_k.
	\end{align*}
	It follows from this explicit representation that the above isomorphisms are consistent. Moreover, by Khintchine's inequality there exist $C_p > 0$ such that for all $n \in \IN$, $a_1, \ldots, a_n \in \IC$ 
		\[ C_p^{-1} \biggl( \sum_{k=1}^n \abs{a_k}^2 \biggr)^{1/2} \le \biggnorm{\sum_{k=1}^n a_k r_k}_{L^p([0,1])} \le C_p \biggl( \sum_{k=1}^n \abs{a_k}^2 \biggr)^{1/2}. \]
	Therefore the isomorphisms are uniformly bounded in $n$. Hence, one has consistent isomorphisms $i_p: X^p = \oplus_{\ell^p}^n \ell^2_n \xrightarrow{\sim} \oplus_{\ell^p}^n \Rad_n$. The right hand side is a subspace of $\oplus_{\ell^p}^n L^p([0,1]) \simeq L^p((0,\infty))$. We now show that there exist consistent projections $Q^p$ from $L^p((0, \infty))$ onto this subspace. Indeed, for a fixed $p \in (1,\infty)$ the space $\Rad_n$ in $L^p([0,1])$ is uniformly complementable in $n$~\cite[Theorem~1.12(c) and comments after the proof]{DJT95}, where the projections explicitly given by
		\[ P_n: f \mapsto \sum_{k=1}^n r_k \int_0^1 f(\omega) r_k(\omega) \, d\omega \]
	are consistent for all $p$. By the uniform complementability one obtains consistent projections $Q^p$ as desired. From these we obtain consistent decompositions $L^p((0,\infty)) \simeq (\oplus_{\ell^p}^n \Rad_n) \oplus Z_p$. Let $(T_p(z))_{z \in \Sigma_{\pi / 2}}$ be the family of semigroups obtained from Proposition~\ref{prop:concrete_bases}. Using the above decomposition we can define consistent holomorphic semigroups $(S_p(z))_{z \in \Sigma_{\pi/2}}$ ($p \in (1, \infty)$) on $L^p((0, \infty))$ by
		\[ S_p(z) \coloneqq i_p \circ T_p(z) \circ i_p^{-1} \oplus \Id. \]
	Clearly, $(S_p(z))_{z \in \Sigma_{\pi/2}}$ has maximal regularity if and only if $(T_p(z))_{z \in \Sigma_{\pi/2}}$ has maximal regularity. Hence, by Proposition~\ref{prop:concrete_bases} $(S_p(z))_{z \in \Sigma_{\pi/2}}$ does not have maximal regularity for $p \in (2, \infty)$.
\end{proof}

We can now easily modify the above counterexample to obtain the main result of this section.

\begin{corollary}[Maximal Regularity does not extrapolate]\label{cor:lp_counterexample+} There exist consistent holomorphic $C_0$-semigroups $(R_p(z))_{z \in \Sigma_{\pi/2}}$ on $L^p(\IR)$ for $p \in (1, \infty)$ such that $(R_p(z))_{z \in \Sigma_{\pi/2}}$ has maximal regularity iff $p = 2$. 
\end{corollary}
\begin{proof}
	Let $(S_p^*(z))_{z \in \Sigma_{\pi/2}}$ be the adjoint semigroups (which are again strongly continuous and holomorphic) of the semigroups $(S_p(z))_{z \in \Sigma_{\pi/2}}$ constructed in Theorem~\ref{thm:l^p_counterexample}. By~\cite[Corollary~2.11]{KunWei04}, $(S_p^*(z))_{z \in \Sigma_{\pi/2}}$ does not have maximal regularity for $p \in (1,2)$. Now, the direct sum $R_p(z) = S_p(z) \oplus S_p^*(z)$ of both has the desired properties.
\end{proof}

\begin{remark}
	Maximal regularity extrapolates from $(T_2(t))_{t \ge 0}$ on $L^2$ to $(T_p(t))_{t \ge 0}$ on $L^p$ under the additional assumption $\R{T_p(t): 0 < t < 1} < \infty$ by~\cite[Theorem~4.1]{AreBu03} or~\cite[Corollary~6.2]{Fac13+}. The above counterexample shows that we cannot omit this assumption.
\end{remark}

We now return to the basis $f_m^{\prime}$ obtained from Proposition~\ref{prop:concrete_bases}. In the case $p \in (1,2)$ left open before we have the following technical result.

\begin{proposition} The basis $f_m^{\prime}$ constructed in Proposition~\ref{prop:concrete_bases} is unconditional for $p \in (1, 2]$. \end{proposition}
\begin{proof} 
	Assume that $\sum_{m=1}^{\infty} a_m f_m^{\prime}$ converges. We have to show that $\sum_{m=1}^{\infty} \epsilon_m a_m f_m^{\prime}$ converges for any choice of signs $(\epsilon_m)_{m \in \IN} \in \{-1,1\}^{\IN}$. Obviously, the even part $\sum_{m=1}^{\infty} \epsilon_{2m} a_{2m} e_{\pi(2m)}$ converges as there is no interference between two different components of $(e_m)_{m \in \IN}$. For the odd part we have to check the convergence of the series
	 	\begin{align*}
			\MoveEqLeft \biggl( \sum_{k=1}^{\infty} \biggl( \sum_{\substack{l \in B_k \\ l \text{ odd}}} \abs{\epsilon_l a_l + \epsilon_{l+1} a_{l+1}}^2 \biggr)^{p/2} \biggr)^{1/p} \le \biggl( \sum_{k=1}^{\infty} \biggl( \sum_{\substack{l \in B_k \\ l \text{ odd}}} \abs{a_l + a_{l+1}}^2 \biggr)^{p/2} \biggr)^{1/p} \\
			& + \biggl( \sum_{k=1}^{\infty} \biggl( \sum_{\substack{l \in B_k \\ l \text{ odd}}} \abs{(\epsilon_{l+1} - \epsilon_l)a_{l+1}}^2 \biggr)^{p/2} \biggr)^{1/p}
		\end{align*} 
		
	The first term converges by assumption. Again, we split the second term. First notice that $(a_{4m+2})_{m \in \IN} \in \ell^p$. Observe that for $p \le 2$ we have the inclusion $\ell^p \hookrightarrow \ell^2$ which yields $\ell^p \hookrightarrow \oplus_{\ell^p}^n \ell^2_n$. From this inclusion we deduce that the part of the second term where $l$ runs over the numbers $l$ with $l+1 \equiv 2 \mod 4$ converges. Finally, we have to show that the part of the second term where $l$ runs over the numbers with $l+1 \equiv 0 \mod 4$ converges. This part is built from the convergent series
		\[ \sum_{m=1}^{\infty} a_{4m} e_{\pi(4m)} \]
	by eventually inserting zeros. The following lemma shows that this procedure does not destroy the convergence and finishes the proof.
\end{proof}

Let $(a_m)_{m \in \IN}$ be a sequence and $(b_m) = (0, \ldots, 0, a_1, 0, \ldots, 0, a_2, \ldots)$ be a sequence built from $(a_m)_{m \in \IN}$ by inserting zeros. We can then introduce a mapping $\phi: \IN \to \IN$ which maps $k$ to the position of $a_k$ in the new sequence $(b_m)_{m \in \IN}$.

\begin{lemma}\label{lem:technical_lemma} Let $(a_m)_{m \in \IN} \in X^p$ ($1 \le p < \infty$), $(b_m)_{m \in \IN}$ and $\phi: \IN \to \IN$ be as above and suppose that 
	\[ M \coloneqq \sup_{k \in \IN} \phi(k+1) - \phi(k) < \infty. \]
	Then $(b_m)_{m \in \IN} \in X^p$ as well. 
\end{lemma}

\begin{proof} 
	It suffices to consider the case $M = 2$ as the general case then follows inductively. Let $\mathcal{B} \coloneqq \{ B_n: n \in \IN \}$ be the set of all blocks. By considering the worst cases, one sees that for each $A \in \mathcal{B}$ there exist at most three different $B \in \mathcal{B}$ such that $\phi(A) \cap B \neq  \emptyset$ and likewise for each $B \in \mathcal{B}$ there exist at most three different $A \in \mathcal{B}$ such that $\phi(A) \cap B \neq \emptyset$. Choose $C \ge 1$ such that for every $\alpha, \beta, \gamma \in \IR$ one has $(\alpha^2 + \beta^2 + \gamma^2)^{1/2} \le C (\alpha^p + \beta^p + \gamma^p)^{1/p}$. Then for each $B \in \mathcal{B}$ one has
		\begin{align*}
			\biggl( \sum_{m \in B} \abs{b_m}^2 \biggr)^{1/2} \le \biggl( \sum_{\substack{A \in \mathcal{B} \\ \phi(A) \cap B \neq \emptyset}} \sum_{m \in A}\abs{a_m}^2 \biggr)^{1/2} \le C \biggl( \sum_{\substack{A \in \mathcal{B} \\ \phi(A) \cap B \neq \emptyset}} \biggl( \sum_{m \in A} \abs{a_m}^2 \biggr)^{p/2} \biggr)^{1/p}
		\end{align*}
	Therefore one obtains
		\begin{align*}
			\norm{(b_m)}_{X^p}^p & = \sum_{B \in \mathcal{B}} \biggl( \sum_{m \in B} \abs{b_m}^2 \biggr)^{p/2} \le C^p \sum_{B \in \mathcal{B}} \sum_{\substack{A \in \mathcal{B} \\ \phi(A) \cap B \neq \emptyset}} \biggl( \sum_{m \in A} \abs{a_m}^2 \biggr)^{p/2} \\
			& = C^p \sum_{A \in \mathcal{B}} \sum_{\substack{B \in \mathcal{B} \\ \phi(A) \cap B \neq \emptyset}} \bigg( \sum_{m \in A} \abs{a_m}^2 \biggr)^{p/2} \le 3C^p \sum_{A \in \mathcal{B}} \biggl( \sum_{m \in A} \abs{a_m}^2 \biggr)^{p/2}. \qedhere
		\end{align*}
\end{proof}

\begin{corollary} There exist consistent holomorphic $C_0$-semigroups $(T_p(z))_{z \in \Sigma_{\pi/2}}$ on $L^p((0, \infty))$ for $p \in (1, \infty)$ such that $(T_p(z))_{z \in \Sigma_{\pi/2}}$ has maximal regularity iff $p \in (1,2]$ (iff $p \in [2,\infty)$). 
\end{corollary}

\section{Schauder bases that are not \texorpdfstring{$\mathcal{R}$}{R}-bases}

In this section we show that our methods also directly give examples of bases that are not $\mathcal{R}$-bases.
Let $(e_m)_{m \in \IN}$ be a Schauder basis for a Banach space $X$. Then one can define projections $P_N: X \to X$ ($N \in \IN$) by
	\[ P_N \biggl(\sum_{m=1}^{\infty} a_m e_m \biggr)  = \sum_{m=1}^N a_m e_m. \]
It follows from the uniform boundedness principle that the projections are always uniformly bounded in operator norm~\cite[Proposition~1.1.4]{AlbKal06}. This motivates the following definition.
	
\begin{definition}
	A Schauder basis $(e_m)_{m \in \IN}$ is called an \emph{$\mathcal{R}$-basis} if the set $\{ P_N: N \in \IN \}$ of projections is $\mathcal{R}$-bounded.
\end{definition}

One can show that each unconditional basis on a UMD-space (this actually holds for spaces with property $(\Delta)$~\cite[Theorem~3.3~(4)]{KalWei01}) is  an $\mathcal{R}$-basis, in particular this holds for $L^p$-spaces for $1 < p < \infty$.

\begin{theorem}
	Let $X$ be a Banach space that admits an unconditional, non-symmetric normalized Schauder basis $(e_m)_{m \in \IN}$. Then there exists a Schauder basis for $X$ that is not an $\mathcal{R}$-basis.
\end{theorem}
\begin{proof}
	As in the proof of Proposition~\ref{prop:ce_constructor} we choose $f_m^{\prime}$ respectively $f_m^{\prime \prime}$. Again, we only consider the case of $f_m^{\prime}$. Suppose that $(f_m^{\prime})_{m \in \IN}$ is an $\mathcal{R}$-basis and let $(P_N)_{N \in \IN}$ be the associated projections. Then $\sum_{m=1}^n r_m x_m \mapsto \sum_{m=1}^n r_m P_{2m-1} x_m$ extends to a bounded operator $\tilde{P} \in \mathcal{L}(\Rad(X))$. In particular, we have
		\begin{align*} 
			-\sum_{m=1}^N r_m a_m e_{2m-1} & = -\sum_{m=1}^N r_m a_m f^{\prime}_{2m-1} = \sum_{m=1}^N r_m a_m P_{2m-1} (f_{2m}^{\prime} - f_{2m-1}^{\prime}) \\
			& = \sum_{m=1}^N r_m a_m P_{2m-1} e_{\pi(2m)} = \tilde{P} \biggl( \sum_{m=1}^N r_m a_m e_{\pi(2m)} \biggr).
		\end{align*}
	The boundedness of $\tilde{P}$ implies that the left hand side converges in $\Rad(X)$ whenever $\sum_{m=1}^N r_m a_m e_{\pi(2m)}$ converges for $N \to \infty$, which exactly as in the proof of Proposition~\ref{prop:ce_constructor} yields a contradiction.
\end{proof}

In particular, this applies for the spaces $L^p([0,1])$ $(1 < p \neq 2 < \infty)$ and answers an open problem stated at the end of~\cite{KriLeM10}.

\bibliographystyle{amsalpha}
\bibliography{counterexample_maxreg}{}

\end{document}